\documentclass[11pt,reqno]{amsart}
 \usepackage{amssymb,amsfonts,amscd,amsbsy}
\usepackage[mathscr]{eucal}
\usepackage{url}

\newtheorem{thm}{Theorem}[section]

\theoremstyle{definition}

\numberwithin{equation}{section}

\makeatletter
\def\imod#1{\allowbreak\mkern5mu({\operator@font mod}\,\,#1)}
\makeatother

\begin{document}

\title[On two 10th order mock theta identities]
{On two 10th order mock theta identities} 
 
\author{Jeremy Lovejoy and Robert Osburn}

\address{CNRS, LIAFA, Universit{\'e} Denis Diderot - Paris 7, Case 7014, 75205 Paris Cedex 13, FRANCE}

\address{School of Mathematical Sciences, University College Dublin, Belfield, Dublin 4, Ireland}

\email{lovejoy@liafa.jussieu.fr}

\email{robert.osburn@ucd.ie}


\date{\today}

\begin{abstract}
We give short proofs of conjectural identities due to Gordon and McIntosh involving two 10th order mock theta functions.
\end{abstract}
 
\maketitle

\section{Introduction}

Two of the ``10th order" mock theta functions found on page 9 of Ramanujan's lost notebook \cite{Ra1} are

$$
X(q):= \sum_{n \geq 0} \frac{(-1)^n q^{n^2}}{(-q; q)_{2n}}
$$

\noindent and

$$
\chi(q):= \sum_{n \geq 0} \frac{(-1)^n q^{(n+1)^2}}{(-q; q)_{2n+1}}.
$$

\noindent Here and throughout, we use the standard $q$-hypergeometric notation

\begin{equation*}
(a)_n = (a;q)_n = \prod_{k=1}^{n} (1-aq^{k-1}),
\end{equation*}

\noindent valid for $n \in \mathbb{N} \cup \{\infty\}$. In the spirit of the celebrated ``5th order" mock theta conjectures \cite{AG,Hi1}, Gordon and McIntosh have recently conjectured the following identities for $X(q)$ and $\chi(q)$ (see (5.18) in \cite{GM1} or \cite{GM2}):

\begin{equation} \label{mtc3}
\begin{aligned}
X(-q^2) & = - 2q g_2(q, q^{20}) + 2q^5 g_2(q^9, q^{20}) \\
& + \frac{(q^4; q^4)_{\infty}^2 \bigl( j(-q^2, q^{20})^2 j(q^{12}, q^{40}) + 2q (q^{40}; q^{40})_{\infty}^3 \bigr)}{(q^2; q^2)_{\infty} (q^{20}; q^{20})_{\infty} (q^{40}; q^{40})_{\infty} j(q^8, q^{40})} \\
\end{aligned}
\end{equation}

\noindent and

\begin{equation} \label{mtc4}
\begin{aligned}
\chi(-q^2) & = - 2q^3 g_2(q^3, q^{20}) - 2q^5 g_2(q^7, q^{20}) \\
& + \frac{q^2 (q^4; q^4)_{\infty}^2 \bigl( 2q (q^{40}; q^{40})_{\infty}^3 - j(-q^6, q^{20})^2 j(q^4, q^{40}) \bigr)}{(q^2; q^2)_{\infty} (q^{20}; q^{20})_{\infty} (q^{40}; q^{40})_{\infty} j(q^{16}, q^{40})}, \\
\end{aligned}
\end{equation}

\noindent where $g_2(x,q)$ is a so-called ``universal mock theta function"

$$
g_2(x, q):= \frac{1}{j(q, q^2)} \sum_{n \in \mathbb{Z}} \frac{(-1)^n q^{n(n+1)}}{1-xq^n}
$$
\noindent and $j(x,q):=(x)_{\infty} (q/x)_{\infty} (q)_{\infty}.$

As stated in \cite{GM2}, identities (\ref{mtc3}) and (\ref{mtc4}) were discovered by using computer algebra and ``a rigorous proof has yet to be worked out".  It is now well-known that such mock theta conjectures can be reduced to a finite (but possibly formidable) computation by using the theory of harmonic weak Maass forms (see \cite{folsom} for an example of such an argument). Instead of following this approach, we observe that these two identities follow easily upon appealing to results of Choi \cite{choi2} which express $X(q)$ and $\chi(q)$ in terms of Appell-Lerch sums and modular forms, applying properties of Appell-Lerch sums and verifying a simple modular form identity.  Since all classical mock theta functions can be written in terms of Appell-Lerch sums (see Section 4 of \cite{Hi-Mo1}), one can also easily prove identities similar to (\ref{mtc3}) and (\ref{mtc4}) for 2nd, 3rd, 6th and 8th order mock theta functions, see (5.2), (3.12), (5.10) and the top of page 125 in \cite{GM1}. The details are left to the interested reader. Our main result is the following.

\begin{thm} \label{main} Identities (\ref{mtc3}) and (\ref{mtc4}) are true.
\end{thm}

\section{Preliminaries}

We recall some required facts which are conveniently given in \cite{Hi-Mo1}. For $x$, $z \in \mathbb{C}^{*}:=\mathbb{C} \setminus \{ 0 \}$ with neither $z$ nor $xz$ an integral power of $q$, define the Appell-Lerch sums

\begin{equation*} 
m(x,q,z) := \frac{1}{j(z,q)} \sum_{r \in \mathbb{Z}} \frac{(-1)^r q^{\binom{r}{2}} z^r}{1-q^{r-1} xz}.
\end{equation*}

\noindent Also, if $x^2$ is neither zero nor an integral power of $q^2$, then define

\begin{equation*}
k(x,q):= \frac{1}{xj(-q,q^4)} \sum_{n \in \mathbb{Z}} \frac{q^{n(2n+1)}}{1-q^{2n} x^2}.
\end{equation*}

Following \cite{Hi-Mo1}, we use the term ``generic" to mean that the parameters do not cause poles in the the Appell-Lerch sums or in the quotients of theta functions. For generic $x$, $z$, $z_0$ and $z_1 \in \mathbb{C}^{*}$, the sums $m(x,q,z)$ satisfy (see (2.2b) of Proposition 2.1, Theorem 2.3, Corollary 2.7, Proposition 3.4 and Proposition 3.6 in \cite{Hi-Mo1})

\begin{equation} \label{m1}
m(x,q,z)=x^{-1} m(x^{-1}, q, z^{-1}),
\end{equation}

\begin{equation} \label{m2}
m(x, q, z_1) = m(x, q, z_0) + \Delta(x, q, z_1, z_0),
\end{equation}

\begin{equation} \label{m3}
m(x,q,z) = m(-qx^2, q^4, z^4) - q^{-1} x m(-q^{-1}x^2, q^4, z^4) - \Lambda(x,q,z),
\end{equation}

\begin{equation} \label{m4}
g_2(x,q) = -x^{-1} m(x^{-2}q, q^2, x)
\end{equation} 

\noindent and

\begin{equation} \label{m5}
xk(x,q) = m(-x^2, q, x^{-2}) + \frac{J_1^4}{2J_2^2 j(x^2, q)}.
\end{equation}

\noindent Here

$$
 \Delta(x, q, z_1, z_0) := \frac{z_0 J_1^3 j(z_{1} / z_{0}, q) j(xz_{0} z_{1}, q)}{j(z_{0}, q) j(z_{1}, q) j(xz_{0}, q) j(xz_{1}, q)},
$$

$$
\Lambda(x,q,z):=\frac{J_2 J_4 j(-xz^2, q) j(-xz^3, q)}{x j(xz, q) j(z^4, q^4) j(-qx^2 z^4, q^2)}
$$

\noindent and $J_{m}:= J_{m, 3m}$ with $J_{a,m}=J_{a,m}(q):= j(q^{a}, q^{m}).$

\section{Proof of Theorem \ref{main}} 
\begin{proof}[Proof of Theorem \ref{main}]

As observed in \cite{Hi-Mo1}, after replacing $q$ with $-q^2$ in $X(q)$ on page 183 of \cite{choi2} and using (\ref{m5}) followed by (\ref{m2}), we have 

\begin{equation} \label{s1}
\begin{aligned}
X(-q^2) & = -2q^2 k(-q^2, -q^{10}) - \frac{J_5(-q^2) J_{10}(-q^2) J_{2,5}(-q^2)}{J_{2,10}(-q^2) J_{1,5}(-q^2)}\\
&=2m(-q^4, -q^{10}, q^8) - \frac{J_{3,10}(-q^2) J_{5,10}(-q^2)}{J_{1,5}(-q^2)}.
\end{aligned}
\end{equation}

\noindent Now, taking $x=-q^4$, $q \to -q^{10}$ and $z=q^8$ in (\ref{m3}) and simplifying via (\ref{m1}) yields

\begin{equation} \label{s2}
m(-q^4, -q^{10}, q^8) = m(q^{18}, q^{40}, q^{32}) -q^{-4} m(q^2, q^{40}, q^{-32}).
\end{equation}

\noindent Here, $\Lambda(-q^4, -q^{10}, q^8)=0$. By (\ref{m2}), we have

\begin{equation} \label{s3}
m(q^{18}, q^{40}, q^{32}) = m(q^{18}, q^{40}, q) + \Delta(q^{18}, q^{40}, q^{32}, q)
\end{equation}

\noindent and

\begin{equation} \label{s4}
m(q^2, q^{40}, q^{-32}) = m(q^2, q^{40}, q^9) + \Delta(q^2, q^{40}, q^{-32}, q^9).
\end{equation}

\noindent Comparing (\ref{mtc3}) with (\ref{m4}) and (\ref{s1})--(\ref{s4}), it now suffices to prove the identity

\begin{equation*}
\begin{aligned}
& \frac{(q^4; q^4)_{\infty}^2 \bigl( j(-q^2, q^{20})^2 j(q^{12}, q^{40}) + 2q (q^{40}; q^{40})_{\infty}^3 \bigr)}{(q^2; q^2)_{\infty} (q^{20}; q^{20})_{\infty} (q^{40}; q^{40})_{\infty} j(q^8, q^{40})}  \\
& = - \frac{J_{3,10}(-q^2) J_{5,10}(-q^2)}{J_{1,5}(-q^2)} + 2 \Delta(q^{18}, q^{40}, q^{32}, q) - 2q^{-4} \Delta(q^2, q^{40}, q^{-32}, q^9).
\end{aligned}
\end{equation*}

\noindent This identity has been verified using Garvan's MAPLE program, see 
\begin{center}
\url{http://www.math.ufl.edu/~fgarvan/qmaple/theta-supplement/} 
\end{center}
\noindent This proves (\ref{mtc3}). 

After replacing $q$ with $-q^2$ in $\chi(q)$ on page 183 of \cite{choi2} and again using (\ref{m5}) followed by (\ref{m2}), we have 

\begin{equation} \label{s6}
\begin{aligned}
\chi(-q^2) & = 2 - 2q^4 k(q^4, -q^{10}) - q^2 \frac{J_5(-q^2) J_{10}(-q^2) J_{1,5}(-q^2)}{J_{4,10}(-q^2) J_{2,5}(-q^2)} \\
& = 2m(q^2, -q^{10}, q^4) - \frac{q^2 J_{1,10}(-q^2) J_{5,10}(-q^2)}{J_{2,5}(-q^2)}.
\end{aligned}
\end{equation}

\noindent Now, taking $x=q^2$, $q \to -q^{10}$ and $z=q^4$ in (\ref{m3}) and simplifying via (\ref{m1}) yields

\begin{equation} \label{s7}
m(q^2, -q^{10}, q^4) = m(q^{14}, q^{40}, q^{16}) + q^{-2} m(q^6, q^{40}, q^{-16}).
\end{equation}

\noindent Here, $\Lambda(q^2, -q^{10}, q^4)=0$. By (\ref{m2}), we have

\begin{equation} \label{s8}
m(q^{14}, q^{40}, q^{16}) = m(q^{14}, q^{40}, q^3) + \Delta(q^{14}, q^{40}, q^{16}, q^3)
\end{equation}

\noindent and

\begin{equation} \label{s9}
m(q^6, q^{40}, q^{-16}) = m(q^6, q^{40}, q^7) + \Delta(q^6, q^{40}, q^{-16}, q^7).
\end{equation}

\noindent Comparing (\ref{mtc4}) with (\ref{m4}) and (\ref{s6})--(\ref{s9}), it now suffices to prove the identity

\begin{equation*}
\begin{aligned}
& \frac{(q^4; q^4)_{\infty}^2 \bigl( 2q (q^{40}; q^{40})_{\infty}^3 - j(-q^6, q^{20})^2 j(q^4, q^{40}) \bigr)}{(q^2; q^2)_{\infty} (q^{20}; q^{20})_{\infty} (q^{40}; q^{40})_{\infty} j(q^{16}, q^{40})} \\
& = - q^2\frac{J_{1,10}(-q^2) J_{5,10}(-q^2)}{J_{2,5}(-q^2)} + 2\Delta(q^{14}, q^{40}, q^{16}, q^3) + 2q^{-2} \Delta(q^6, q^{40}, q^{-16}, q^7).
\end{aligned}
\end{equation*}

\noindent This identity has also been verified using Garvan's MAPLE program. This proves (\ref{mtc4}).

\end{proof}

\section*{Acknowledgements} The second author would like to thank the Institut des Hautes {\'E}tudes
Scientifiques for their support during the preparation of this paper.


\begin{thebibliography}{99}

\bibitem{AG}
G. E. Andrews and F. Garvan, \emph{Ramanujan ``lost" notebook. VI. The mock theta conjectures}, Adv. in Math. \textbf{73} (1989), no. 2, 242--255. 

\bibitem{choi2}
Y.-S. Choi, \emph{Tenth order mock theta functions in Ramanujan's lost notebook II}, Adv. in Math. \textbf{156} (2000), no. 2, 180--285.

\bibitem{folsom}
A. Folsom, \emph{A short proof of the mock theta conjectures using Maass forms}, Proc. Amer. Math. Soc. \textbf{136} (2008), no. 12, 4143--4149.

\bibitem{GM1}
B. Gordon and R.J. McIntosh, \emph{A survey of classical mock theta functions}, in: Partitions, $q$-series, and Modular Forms, Developments in Mathematics 2012, vol. 23, 95--144.

\bibitem{Hi1}
D. Hickerson, \emph{A proof of the mock theta conjectures}, Invent. Math. {\bf 94} (1988), no. 3, 639--660.

\bibitem{Hi-Mo1} 
D. Hickerson and E. Mortenson, \emph{Hecke-type double sums, Appell-Lerch sums, and mock theta functions (I)}, preprint available at \url{http://arxiv.org/abs/1208.1421}

\bibitem{GM2}
R.J. McIntosh, \emph{The H and K family of mock theta functions}, Canad. J. Math. \textbf{64} (2012), no. 4, 935--960.

\bibitem{Ra1}
S. Ramanujan, \emph{The Lost Notebook and Other Unpublished Papers}, Narosa Publishing House, New Delhi, 1988.

\end{thebibliography}
\end{document}